\documentclass[12pt,reqno]{amsart}
\usepackage{tikz,dsfont}
\usepackage{amsmath, amssymb, graphicx, mathrsfs, hyperref}
\usepackage{verbatim} 
\usepackage{enumerate}
\newtheoremstyle{dotless}{}{}{\itshape}{}{\bfseries}{}{ }{} 
\theoremstyle{dotless}

\newtheorem{thm}{Theorem}
\newtheorem{lemma}{Lemma}

\newtheorem{prop}{Proposition}
\newtheorem{conj}{Conjecture}

\newtheorem{defn}{Definition}

\newcommand{\Mod}[1]{\ (\mathrm{mod}\ #1)}

\begin{document}

\pagenumbering{arabic} \setcounter{page}{1}

\title{On the largest sum-free subset problem in the integers}

\author{George Shakan}

 \thanks{The author thanks Ben Green for useful discussions.}
\address{Department of Mathematics \\
Universit\'e de Montr\'eal }    
\email{george.shakan@gmail.com}

\begin{abstract}
Let $A \subset \mathbb{Z}_{>0}$ of size $n$. It is conjectured that for any $C >0$ and $n$ large enough that $A$ contains a sum-free subset of size at least $n/3 +C$. We study this problem and find an alternate proof of Bourgain's result that one make take $C=2/3$. 
\end{abstract}

\maketitle

\section{Introduction}

Let $A \subset \mathbb{Z}_{> 0}$ be a set of size $n$. We say $A$ is {\it sum-free} if there are no solutions to 
$$x+y = z,$$
with $x,y,z \in A$. Let $s(A)$ be the largest sum-free subset of $A$. Erd\H{o}s \cite{Er} observed that $s(A) \geq n/3$, Alon and Kleitman \cite{AK} remarked that his argument proves $s(A) \geq (n+1)/3$. Bourgain later proved the following.

\begin{thm}[\cite{B}]\label{2} Let $A \subset \mathbb{Z}_{>0}$ be a set with co-prime elements of size $n$. Then either $A = \{1,2\}$ or 
$$s(A) \geq (n+2)/3.$$
\end{thm}
The main question in the area is the following.

\begin{conj}\label{MC} 
Let $A \subset \mathbb{Z}_{>0}$ of size $n$ and $S$ be a positive integer. Then there is a $N_S \in \mathbb{Z}$ such that
$$s(A) \geq (n+S)/3,$$
for all $n \geq N_S$.  
\end{conj} 
Bourgain \cite{B} also showed that there is a $C > 0$ such that 
\begin{equation}\label{l1} s(A) \geq n/3 + C ||\sum_{a \in A} \cos(2\pi a \cdot)||_{L^1(\mathbb{T})},\end{equation}
answering Conjecture~\ref{MC} for a ``typical" set $A$. In light of Bourgain's proof of Theorem~\ref{2}, we let 
$$f(x) = 1_{[1/3 , 2/3)} - 1/3,$$
and 
$$f_A(x) = \sum_{a \in A} f(ax).$$
We also set 
$$m_A := \max_{x \in \mathbb{R}} f_A(x).$$
We have the following stronger conjecture.

\begin{conj}\label{MC2} 
Let $A \subset \mathbb{Z}_{>0}$ of size $n$ and $S$ be a positive integer. Then there is a $N_S \in \mathbb{Z}$ such that
$$m_A \geq S/3,$$
for all $n \geq N_S$.  
\end{conj} 

Our first result shows that in order to establish Conjecture~\ref{MC2}, it is enough to establish it for finitely many values of $n$. 

\begin{thm}\label{main} 
Suppose that Conjecture~\ref{MC2} is true for all $n$ with $N_S <  n \leq 3N_S + 2S$. Then Conjecture~\ref{MC2} is true for all $n \geq N_S$. 
\end{thm}

Thus it remains to understand what happens for small values of $n$. First we ask, for a fixed $n$, is determining that $m_A \geq S/3$ a finite computation? We cannot answer this, but by replacing $f$ in Conjecture~\ref{MC2} with a suitable minorant, we can. We will use the Selberg polynomials (the exact choice is not important), which we recall in \eqref{SK} below. The important properties of $S_K$ is that it is a minorant of $f$, its Fourier support is contained in $\{-K , \ldots , K\}$ and $\int (f - S_K ) \leq \frac{1}{K+1}$. 

\begin{thm}\label{main2}
Let  $\delta >0$, and $C = C(\delta)$ be large enough. Suppose for some $K\geq 1$
$$\max_{x\in \mathbb{T}} \sum_{a\in A} S_K(ax) \geq S/3 + \delta,$$
for all sets $A$ of size $n$ contained in $\{1 , \ldots , T\}$ where $T = \exp(O_{\delta}(n\log (nK)))$. Then Conjecture~\ref{MC2} is true for sets of size $n$ and the same value of $S$.
\end{thm}

The discretized nature of $f_A$ ensures that $m_A \geq S/3$ and $m_A \geq S/3 + \delta$ are not different conditions for $\delta <1/3$. However, this is not the case in Theorem~\ref{main2}. 

In light of Theorem~\ref{main}, it is natural to ask whether we can understand the behavior of $m_A$ for small values of $n$. For instance the first open case is whether $m_A \geq 1$ for all $A$ of large enough size. As $m_A$ is an integer, by Theorem~\ref{2} one only has to check for sets that have size $= 1 \Mod 3$. In order to illustrate some relevant techniques, we classify sets of size 4 which have $m_A \leq 2/3$.

\begin{thm}\label{main4} Let $A \subset \mathbb{Z}_{>0}$ of size 4 with co-prime elements. Then either $m_A > 2/3$,
$$A = \{u , 2u , v , 2v\},$$
or $A$ is one of 
$$\{1,2,3,4\}, \{1,2,4,5\},\{1,4,5,8\}, \{2,3,5,10\}.$$
\end{thm}

On the other hand, we do find an example with $n=10$ such that $m_A = 2/3$. In fact, this comes from work on the Lonely Runner conjecture \cite{BHK}, which we elaborate on in Section~\ref{LRC}. 

We prove Theorem~\ref{main} in Section~\ref{sec2}. In Section~\ref{sec3}, we provide a computation that computes the Legesgue measure of 
$$\{\theta \in \mathbb{R} / \mathbb{Z} : 1/3 \leq ||\theta a || , ||\theta b || < 2/3\}.$$
In Section~\ref{sec4}, we give proof of Bourgain's Theorem~\ref{2}, utilizing Theorem~\ref{main}. In Section~\ref{sec5}, we characterize when $m_A \leq 1$ for sets of size 3. We then use this result to provide a similar characterization for sets of size 4 in Section~\ref{sec6}. In Section~\ref{LRC}, we mention a connection the the lonely runner conjecture. In section~\ref{finite}, we prove Theorem~\ref{main2}.

\section{Proof of Theorem~\ref{main}}\label{sec2}

We let $A = A_0 \cup A_1$ where $A_0$ are the elements of $A$ co-prime to $3$ and $A_1$ are the elements that are divisible by 3. We have the following easy lemma, which establishes Theorem~\ref{main} if $A_0$ is large. 

\begin{lemma}\label{lem1} 
With the notation above,
$$m_A \geq |A_0|/6 - |A_1|/3.$$
\end{lemma}
\begin{proof}
By pigeon-hole, either at least half the elements of $A_0$ or at least $1 \Mod 3$ or $2\Mod3$. Thus either
$$f_{A_0}(1/3) \geq |A_0|/6,$$
or 
$$f_{A_0}(2/3) \geq |A_0|/6.$$
The result follows as 
$$m_A \geq \max\{ f_{A_0}(1/3) , f_{A_0}(2/3)\} - |A_1|/3.$$

\end{proof}

We now provide the inductive step. 

\begin{lemma}\label{lem2} With the notation above, 
$$m_A \geq m_{A_1}.$$
\end{lemma}
\begin{proof}
We let $x \in \mathbb{T}$ such that 
$$f_{A_1} (x) = m_{A_1}.$$
As the elements of $A$ are divisible by 3, we have
$$f_{A_1}(x) = f_{A_1}(x+1/3) = f_{A_1}(x+2/3) .$$
Summing over $a \in A_0$, we find 
$$f_A(x)  + f_A(x+1/3) + f_A(x+2/3) = 3 m_{A_1} + \sum_{a \in A_0} \left(f(ax) + f(a(x+1/3)) + f(a(x+2/3)) \right).$$
On the other hand, for any $a \in A_0$, we have
$$f(ax) + f(a(x+1/3)) + f(a(x+2/3)) = 0 ,$$
and the result follows.
\end{proof} 

\begin{proof}[Proof of Theorem~\ref{main}]

We may assume that the elements of $A$ are co-prime, so that $A_0 \neq \emptyset$. We proceed by strong induction, where $N_S < n \leq 3N_S + 2S$ are the base cases. By Lemma~\ref{lem1}, we may assume
$$|A_0|/6 - |A_1|/3 \leq S/3.$$
Using that $|A_0| + |A_1| = n$, we simplify this to 
\begin{equation}\label{NS} \frac{n - 2S}{3} \leq |A_1|.\end{equation}
By our assumption in Theorem~\ref{main}, we may assume $n > 3N_S + 2S$ and combining with \eqref{NS}, we find 
$$N_S < |A_1|.$$
In particular $N_S < |A_1| < n$. The result follows from Lemma~\ref{lem2} and the induction hypothesis.

\end{proof}

\section{An explicit computation for two element sets}\label{sec3}

We input the following computation as it will be somewhat useful later. Furthermore, it is worth noting that for two element sets, we have a complete understanding, as in the following proposition. 

\begin{prop}\label{n2} Let $a,b$ be co-prime positive integers. Then
$$\mu(\{x \in \mathbb{T} : 1/3 \leq  ||ax|| , ||bx|| < 2/3 \} )= \frac{1}{9} (1 + \frac{2\chi(ab)}{ab}).$$
\end{prop}
Note that when $a, b = 1,2$, the measure is zero and is otherwise positive. This is all we need in what follows in future sections (and there are shorter proofs of this fact). 
\begin{proof}
It is enough to show
$$\int_0^1 (1/3 + f(ax))(1/3 + f(bx))dx = \frac{1}{9} (1 + \frac{2\chi(ab)}{ab}).$$
Using that $f$ has mean zero, it is enough to show
$$\int_0^1 f(ax) f(bx)dx = \frac{2}{9} \frac{\chi(ab)}{ab}.$$
We expand $f$ in its Fourier series, and use $\int_0^1 \cos (2\pi ux)  \cos (2\pi vx) dx = \frac{1}{2}1_{u = \pm v}$,  so it is enough to show
$$\sum_{n=1}^{\infty} \sum_{m=1}^{\infty} \frac{\chi(n) \chi(m)}{nm} 1_{an = bm} = \frac{4\pi^2}{27}  \frac{\chi(ab)}{ab}.$$
Note that the solutions to $an = bm$ are precisely $(n,m) = d(b,a)$ for $d \in \mathbb{Z}_{>0}$ and so
it remains to establish
$$\sum_{d=1}^{\infty} \frac{\chi(d)^2 }{d^2}= \frac{4\pi^2}{27}  .$$
But by the Euler product representation,
$$\sum_{d=1}^{\infty} \frac{\chi(d)^2}{d^2} = (1-1/3^2) \zeta(2),$$
and so the result follows from $\zeta(2) = \pi^2/6$. 
\end{proof}

\section{Proof of Bourgain's Theorem}\label{sec4}

We now show how to prove Theorem~\ref{2}, using the ideas from Theorem~\ref{main} alongside some of Bourgain's original ideas.  By Proposition~\ref{n2}, is enough to establish Conjecture~\ref{MC} for all $n \geq 3$.
As $\int f_A = 0$, we have $m_A > 0$. Furthermore, since $m_A \in |A|/3 + \mathbb{Z}$, it is enough to establish Theorem~\ref{2} for  $n = 2 \Mod 3$.  By Theorem~\ref{main}, with $N_S = S = 2$, it is enough to consider the cases $n=5$ and $n=8$. 

We first recall a lemma of Bourgain that, in general, handles the case where $1 \notin A$.

\begin{lemma}\label{1inA} Let $A \subset \mathbb{Z}_{>0}$ of coprime elements of size $n$. Suppose $1 \notin A$. Then $m_A > 1/3$.
\end{lemma}

We include his proof for completeness. 
\begin{proof}
It is enough to construct a function $\phi \geq 0$ such that $\int \phi =1$ and $\int \phi f_A > 1/3$. Let $u$ be the smallest element and $v$ be the smallest element such that $u \nmid v$. Let $c(x) = \cos(2\pi x)$. We set 
$$\phi(x) = (1 - c(ux))(1 - c(vx)).$$
Then $\phi \geq 0$ and $\int \phi = 1$. By Plancherel, we have for any $m \in \mathbb{Z}_{>0}$, 
$$\int \phi(x) c(mx) dx =\frac{-\sqrt{3}}{2\pi} \sum_{a \in A} \sum_{n=1}^{\infty}\frac{\chi(n)}{n} 1_{m = na}.$$
If $m = u$, the $na = u$ implies $a < u$ and so there are no possible choices of $a$. If $m = v$, then $a < v$ and so by our definition of $v$, we must have $u | a$. But this implies $u|v$ and so again by our choice of $v$ there is no possible choice of $a$. If $m = v- u$ we can apply a similar argument to the case $m = v$. Lastly, if $u+v = na$ and $n > 1$ we have $a < v$. Thus $u |a$ and so there are no possible choices of $a$. 

Putting this altogether, we find
$$\int \phi f_A = \frac{\sqrt{3}}{2\pi} (2 - \frac{1}{2} 1_A(u+v)) > 1/3.$$
\end{proof}

In what follows we make use of the notation and proof of Theorem~\ref{main}.

\begin{lemma}\label{5}
Let $A \subset\mathbb{Z}_{>0}$ be a set of size $5$. Then $m_A > 1/3$.
\end{lemma}

\begin{proof} By the proof of Theorem~\ref{main}, we may suppose $A_1 = \{v ,2v\}$. By Lemma~\ref{1inA}, we may suppose $1 \in A$. Let $\phi(x) = (1 - (4/3) c(x) + (2/3) c(2x))$. Then $\phi \geq 0$, $\int \phi = 1$ and $$\int \phi f_A = \frac{\sqrt{3}}{2\pi}\left(\frac{5}{3}- \frac{2}{3} 1_A(2)\right).$$
Thus we may suppose $2 \in A$. Let $u$ be the remaining element. We consider 
$$\phi_u(x) = \phi(ux).$$
Then a computation reveals, using $\chi(n) = 0$ if $3| n$, that 
$$\int f_A \phi = \frac{\sqrt{3}}{2\pi} (5/3 + \frac{\chi(u)}{u} - \frac{8\chi(u)}{3u}1_{u = 0 \Mod 2}) > 1/3 ,$$
as $2u \notin A$.
\end{proof} 

Theorem~\ref{2} then follows from the following.

\begin{lemma}
Let $A \subset\mathbb{Z}_{>0}$ be a set of size $8$. Then $m_A > 1/3$.
\end{lemma}
\begin{proof}
Following Lemma~\ref{5}, we may suppose $A_1 = \{v , 2v\}$ and that $1,2 \in A$. Letting $u \neq 1,2,v,2v$ be the smallest such element, we may also suppose $2u \in A$, again following Lemma~\ref{5}. Thus we have described six of the eight elements of $A$.

We now show the remaining two elements are larger than $2v$. Indeed, as in the proof of Theorem~\ref{main}, 
$$f_{A_0}(x) + f_{A_0}(x+1/3) + f_{A_0}(x+2/3) =0 . $$
Let $x = 1/(6v)$. As $f(2vx) = 1/3$, we conclude $m_A > 1/3$ unless 
$$f_{A_0}(x) = f_{A_0}(x+1/3) =f_{A_0}(x+2/3) =0 . $$
As $f(z) = -1/3$ for any $z\in \mathbb{R}_{\geq 0 }$ with $z < 1/3$, we have 
$$f(x) = f(2x) = f(ux) =f(2ux) = -1/3,$$
and so we may assume $a > 2v$ for the remaining two elements of $a \in A$. 

 We now let
$$\phi(x) = (1- c(x)) (1- c(vx)).$$
We first evaluate $$\int f_A(x) c(vx)dx = \frac{- \sqrt{3}}{2\pi} \sum_{n=1}^{\infty} \sum_{a\in A} \frac{\chi(n)}{n}  1_{na = v}.$$ Suppose $na = v$ for some $n \in \mathbb{Z}$ and $a \in A$. Then $a < v$ and so $a \in\{ 1,2,u,2u,v\}$. As $v = 0 \Mod 3$, we have $n = 0 \Mod 3$ when $a \neq 0 \Mod 3$. Thus in this case $\chi(n) = 0$. It follows that the only solution is $a = v$ and $n=1$. 

Performing similar computations, we find 
\begin{align*}\int f_A \phi& = \frac{\sqrt{3}}{2\pi} \left(2 - \frac{(-1)^v\chi(v-1)}{2(v-1)} - \frac{(-1)^v\chi(v+1)}{2(v+1)}  - \frac{1}{2} E_A \right) \\
& =  \frac{\sqrt{3}}{2\pi} \left(2 + \frac{(-1)^{v}}{(v+1)(v-1)} -\frac{1}{2} E_A)\right)  \\
& \geq  \frac{\sqrt{3}}{2\pi} \left(2 - 1/8 -\frac{1}{2} E_A)\right),\end{align*}
where 
$$E_A. =\sum_{j=1}^2 \sum_{n=1}^{\infty} \frac{\chi(n)}{n} (1_{nju = v+1} +  1_{nju = v-1}).$$
Note that we cannot have both $u | (v+1)$ and $u | (v-1)$ as $u > 3$. Thus $E_A \leq 1/2$ and we conclude 
$$\int f_A \phi > 1/3,$$
as desired.

\end{proof}

\section{The case $n=3$}\label{sec5}
When $A$ has 3 elements, $m_A = 1$ or $m_A = 2$. In this section we classify when $m_A =1$, which will be useful in the case $n=4$. We first need a useful lemma that handles the case when two elements share a common factor that is not 2.

\begin{lemma}\label{gcd3}
Let $A = \{u,v,w\} \subset \mathbb{Z}_{> 0}$ be a set with co-prime elements. Suppose $\gcd(u,v) > 2$, $u \neq 2v$ and $v \neq 2u$. Then $m_A > 1$. 
\end{lemma}
\begin{proof}
By Lemma~\ref{n2}, there is a $x \in \mathbb{T}$ such that 
$$f_{\{u,v\}}(x) \geq 4/3.$$
Letting $d = \gcd(u,v)$, and noting $\gcd(w,d) = 1$, we have the points 
$$w(x + j/d),$$
are $d$ points spaced $1/d$ apart. As $1/d \leq 1/3$ one must lie in $[1/3,2/3)$ and the result follows.

\end{proof}

We now need a lemma that will eventually help to handle some small cases. 
 \begin{lemma}\label{size4} 
Let $A \subset \mathbb{Z}_{> 0}$. Suppose $u,v \in A$ with $u < v$ and $v \neq 2u$. Let $w \in A$ such that $w > 2uv$. Then there is a $x \in \mathbb{R}$ such that 
$$1/3 \leq || ux || , || vx|| , ||wx|| < 2/3.$$  
\end{lemma}

\begin{proof}
By Proposition~\ref{n2}, we know there is an $x\in \mathbb{T}$ such that 
$$1/3 \leq ||xu|| , ||xv|| < 2/3.$$
But the set of $\theta$ such that $1/3 \leq u \leq 2/3$ consists of intervals of the form
$$[\frac{3k+1}{3u} , \frac{3k+2}{3u} ),$$
where $k = 0 , \ldots , u-1$. If this intersects an interval of the same form with $v$ in the place of $u$, then the length of the intersection must be at least $1/(3uv)$. The gap between consecutive intervals, with $w$ in the place of $u$ is $2/(3w)$. Thus if $2/(3w) < 1/(3uv)$ there must be an $x$ such that 
$$1/3 \leq ||xu|| , ||xv||, ||xw|| < 2/3.$$

\end{proof}

%
%

We now move onto the general case of sets of size 3. 

\begin{lemma}\label{n3} Let $A  = \{u<v<w\} \subset \mathbb{Z}_{> 0}$ with co-prime elements. Then $m_A = 1$ if and only if
$$A = \{u,v,u+v\},$$
$$A \cap 2 \cdot A \neq \emptyset,$$
or 
$$A = \{1,5,8\} , \{2,3,10\},$$

\end{lemma}

\begin{proof} In order to establish $m_A > 1$, it is enough to construct a function $\phi$ such that $\phi \geq 0$, $\int \phi = 1$ and $\int f_A \phi  > 1$. Let $c(x) : = \cos(2\pi x)$. First note that 
$\phi_1 := (1+c(x))) \geq 0$. It then follows that $$\phi_2(x) := \frac{\phi_1(x)^3}{7} = 1 + \frac{12}{7} c(x) + \frac{6}{7} c(2x) + \frac{2}{7} c(3x) \geq 0.$$
Thus 
$$\phi_3(x) := \phi_2(x +1/2) \geq 0.$$
We let $\phi_u = \phi_3(ux)$ and $\phi_v = \phi_3(vx)$ and 
$$\phi = \phi_u \phi_v.$$
Then $\phi \geq 0$, as $\phi_3 \geq 0$. Also, as $\int \phi_3 = 1$, we have that either $\int \phi = 1$ there is a solution to 
\begin{equation} \label{jukv} ju  = k v,\end{equation}
with $j,k \in \{1,2,3\}$. We first suppose there is no such solution. We let $\widehat{\phi}(n)$ the the coefficient of $c(nx)$ in the Fourier expansion of $\phi$.  Then we have 
$$m_A \geq \int \phi f_A = \frac{-\sqrt{3}}{8\pi} \sum_{j=0}^3 \sum_{k=0}^3 \widehat{\phi_u}(j)\widehat{\phi_v}(k) \sum_{m=1}^{\infty} \frac{\chi(m)}{m} \sum_{a\in A} 1_{\pm ma \pm ju \pm kv = 0}  .$$
We note that typically this sum is $> 1$ and if it is not then we will show this enforces a small linear relations on $u$ and $v$. While the sum is at first sight infinite, only finitely many terms can be non-zero by the condition 
\begin{equation}\label{ab} \pm ma \pm ju \pm kv = 0.\end{equation}
Our next goal is to understand solutions to \eqref{ab}. 

 We first suppose there is a solution to 
$$mu = \pm ju \pm k v.$$
Then it follows that $(u/k) | \gcd(u,v) $ and so by Lemma~\ref{n3}, we have $u\in \{1,2,3,4,6\}$. A symmetric argument works for the equation
$$mv = \pm ju \pm k v.$$
Thus we conclude $\min A > 6$, or the only solutions to \eqref{ab} with $a \in \{u,v\}$ are with $jk = 0$.

Suppose
$$mw = \pm ju \pm kv.$$
Then
$$v < w < \frac{ (|j|+|k|)}{m} v,$$
and so $m < j+k$. Thus $m < 6$. By a similar argument as above, we may suppose $jk \neq 0$. We also can assume $w\neq u+v$ and by size considerations, $w \neq v -u$. Then one can check if there is no non-trivial solution to 
\begin{equation}\label{m} \frac{ju + kv}{m} = \frac{j'u + k'v}{m'},\end{equation}
(where the trivial solutions are from the solutions when $u,v$ are linearly independent vectors in $\mathbb{R}^2$) with $j , k \in \pm \{1,2,3\}$,  then $\int \phi f_A > 1$. 
Thus we have 
$$u = \frac{k' m - k m'}{j m' - j' m } v.$$
Thus by Lemma~\ref{gcd3}, either $m_A > 1$ or  $v = d$ or $v = 2d$ for some $d |( j m' - j' m )$. Thus $u < v \leq 60$ and $w = \pm ju \pm kv$ for some $j , k \in \{-3 . \ldots , 3\}$. It turns out these cases are easy to check by a computer. Indeed, for each $u,v,w$ one can check there is a $j/d$ such that $f_A(j/d) > 1$ by iterating over all the choices of $1 \leq j \leq d-1$ and $5 \leq d \leq 50$.

It now remains to handle the case where $u \leq 6$. First, suppose there is a solution to \eqref{jukv}. Then $v | j\gcd(u,v)$ and so  $v \leq 6$. By Lemma~\ref{size4}, we may suppose $w \leq 36$ and this is easily checked by computer.

Thus we may suppose there is no solution to \eqref{jukv} and $\min A \leq 6$. Everything outlined above still works in this case, except there may be some contribution from $a = u$. In this case $m \geq (v - 3u)/u$. By the computer search outlined above, we may suppose $v > 123$ and so $m \geq 20$. Then from the computations above, $m_A > 1$, as desired. 
\end{proof}

\section{The case $n=4$}\label{sec6}

 We are now ready to prove Theorem~\ref{main4}. 
 
 \begin{proof}[Proof of Theorem~\ref{main4}]
 Let $A = \{u,v,w,x\}$. It is enough to show there is a $A' \subset A$ of size 3 such that $m_{A'} > 1$. First, suppose $1,5,8 \in A$. Then by Lemma~\ref{size4}, either $m_A > 2/3$ unless the remaining element is at most 80. This is easily checked by computer, the only example with $m_A = 2/3$ is $A= \{1,4,5,8\}$. Now suppose $u , 5u , 8u  \in A$ with $u > 1$. We apply Lemma~\ref{n3} to $\{5u , 8u,x\}$ and conclude that either $3u ,13u , 10u,16u,4u,5u/2 \in A$. But the elements of $A$ are coprime, so we must have $u = 2$ and $A = \{2,5,10,16\}$. This case is ruled out as $f_A(2/3) = 5/3$. 
 
 We do a similar argument if $2u,3u,10u \in A$. The only example with $m_A =2/3$ in this case is $A = \{2,3,5,10\}$.
 
  Thus we may now suppose we are not in these two cases. We now suppose $u,v,u+v \in A$ for some choice $u,v\in\mathbb{Z}_{>0}$. Let $w$ be the remaining element. By Lemma~\ref{n3}, applied to $\{u,v,w\}$, we may suppose $w = 2u $ , $w = 2v$, $v = 2u$, $u = 2v$, $u= 2w$, $v = 2w$, $v = u+w$, or $u=v+w$.

We first suppose $w = 2u$. Then we apply Lemma~\ref{n3} to $\{2u , v, u+v\}$. Using that $u , v > 0$ and $u \neq v$, we conclude that (after possibly swapping $u,v$) $A= \{u , 2u, 3u , 4u\}$ or $A = \{u , 2u , 4u , 5u\}$. As the elements of $A$ are coprime, we conclude $A = \{1,2,3,4\}$ or $A = \{1 ,2,4,5\}$. By symmetry, the same argument works when $w = 2v$. 

We now suppose $v = 2u$. Then $A= \{u , 2u , 3u , w\}$ for some $w \in \mathbb{Z}_{> 0}$. We now apply Lemma~\ref{n3} to $\{u , 3u , w\}$. We conclude, by a similar argument as the previous case, that 
$A = \{1,2,3,4\}$, $A = \{1,2,3,6\}$, $A = \{2,3,4,6\}$.  By symmetry, this argument works if $u = 2v$. 

We now suppose $u = 2w$. Then $A =\{w,v,2w , 2w + v\}$. We now apply Lemma~\ref{n3} to $\{w , v , 2w + v\}$. Then $2w = v$ or $w = 2v$. The case $v=2w$ is impossible and if $w= 2v$, we conclude $A =\{1,2,4,5\}$. 

Finally, we suppose $w= v-u$. Lemma~\ref{n3} applied to $\{u,u+v , v-u\}$ implies that $2v = 3u$ or $v = 3u$. We conclude in this case that $A = \{1,2,3,4\}$. By symmetry, this argument works if $w = u-v$.

Thus we conclude that $A' \cap 2 \cdot A' \neq \emptyset$ for any $A' \subset A$ of size 3. From here, it is straightforward to conclude that $A= \{u,2u , v ,2v\}$ for some $u,v \in \mathbb{Z}_{>0}$
 \end{proof}

\section{A connection to the lonely runner conjecture}\label{LRC}

We saw in Theorem~\ref{main4} that a set of the form $\{u,2u, v , 2v\}$ are the only infinite family of obstructions to $m_A \geq 1$ for sets of size $4$. Thus it is natural to investigate what happens for sets of this form with a larger number of elements. It turns out this is related to the lonely runner conjecture. 

\begin{conj}[Lonely Runner]\label{L} Let $A\subset \mathbb{Z}_{>0}$ be a set of size $n$. Then there exists $x \in \mathbb{T}$ such that 
$$\frac{1}{n+1} \leq ||ax|| \leq 1 - \frac{1}{n+1}.$$
\end{conj}

This is an unsolved conjecture, but it is known for $n \leq 6$. Here we remark the following. The biggest difference between Conjecture~\ref{L} and Conjecture~\ref{MC2} is that the size of the interval depends on $n$. However, it is also worth noting that the end points are included, which is a slightly different situation to Conjecture~\ref{MC2}. 

\begin{prop} Let $A = \cup_{j=1}^5 \{x_i , 2x_i\}$ have size 10. Suppose $A$ is not 
$$\{1,2,3,4,5,6,8,9,10,18\}.$$ Then 
$$m_A > 2/3.$$
\end{prop}
\begin{proof}
It suffices to find a $\theta$ such that 
$$1/6 \leq ||\theta x_j|| <5/6.$$
for all $j = 1, \ldots , 5$. This is done by \cite[Thoerem 3]{BHK}.
\end{proof}

We remark that the interested reader can also consult \cite{R}

\section{Proof of Theorem~\ref{main2}}\label{finite}

We introduce an explicit minorant for $1_{[1/3,2/3)}$. Let $K$ be a positive integer. We let $S_K$ be the Selberg polynomial.
\begin{equation}\label{SK} S_K(x) = \frac{1}{3} - B_K(2/3-x) - B_K(x-1/3),\end{equation}
where 
$$B_K(x) = V_K (x) + \frac{1}{2(K+1)} \Delta_{K+1}(x),$$
$$\Delta_K(x) = \sum_{|k| \leq K} (1 - |k|/K) e(kx),$$
and 
\begin{align*}V_K(x) = &\frac{1}{K+1} \sum_{k=1}^K (\frac{k}{K+1} - \frac{1}{2})\Delta_{K+1}(x-\frac{k}{K+1}) \\ &+ \frac{1}{2\pi(K+1)} \sin(2\pi(K+1) x - \frac{1}{2\pi} \Delta_{K+1}(x) \sin(2\pi x).\end{align*}
We have the following lemma proved, for instance, in  \cite[Chapter 1]{Mo}.

\begin{lemma}[Properties of $S_K$]
We have 
\begin{equation}\label{min}  S_K \leq f, \end{equation}

\begin{equation}\label{intSK} \int f - S_K \leq \frac{1}{K+1}, \end{equation}

\begin{equation}\label{inf} ||S_K||_{\infty} = O(1), \end{equation}

\begin{equation}\label{FSK} \widehat{S_K}(n) = 0 , \ \ \text{if} \ \ n \notin \{-K , \ldots , K\}. \end{equation}
Moreover $S_K$ is even. 
\end{lemma}
 We let $m_K = -\inf_{x \in \mathbb{T}} S_K(x) $, by \eqref{inf}, of minimum of $S_K$. Thus we have
$$S_K + m_K \geq 0.$$
 We let 
$$E_{\alpha} = \{x \in \mathbb{T} : \sum_{a \in A} S_K(xa) \geq \alpha\}.$$ 
Then for any $\alpha$, and $k \in \mathbb{Z}_{ > 0}$,
\begin{equation}\label{mx} \max_{x \in \mathbb{T}} \sum_{a\in A} (S_K(ax)  + m_K) \geq \left( \int( \sum_{a\in A} (S_K (ax)+ m_K))^{2k}dx \right)^{1/(2k)} \geq \mu(E_{\alpha})^{1/(2k)} (\alpha + n m_K).\end{equation}
Thus we need a lower bound for $\mu(E_{\alpha})$. 

\begin{lemma}\label{meas}  Let $K, S_K,E_{\alpha}$ be as above and suppose $$\alpha = \left(\max_{x \in \mathbb{T}} \sum_{a\in A} S_K(ax)\right) - \delta.$$ Then there is a $C = C(\delta)$ such that  $\mu(E_{\alpha }) \geq  \exp(- K n  \log (C\delta^{-1}nK))$. 
\end{lemma}
\begin{proof}
Suppose 
$$S_K(\theta) = \max_{x \in \mathbb{T}} \sum_{a\in A} S_K(ax).$$ We will show many translates of $\theta$ lie in $E_{\alpha }$. Indeed letting $\epsilon >0$ we define the Bohr set 
$$B(v, \epsilon) : = \{x \in \mathbb{T} : ||vx|| \leq \epsilon\}.$$
We let 
$$B := \bigcap_{a\in A} \bigcap_{k=1}^K B(ak , \epsilon).$$
It is known \cite[Lemma 4.20]{TV} that 
$$\mu(B) \geq \epsilon^{Kn}.$$
On the other hand, for $t \in B$, we have, using \eqref{inf} and \eqref{FSK}
$$\sum_{a \in A} S_K(a(\theta +t)) \geq \sum_{a \in A}S_K(a \theta) + O(\epsilon n \sum_{k=1}^K |\widehat{S_K}(k)|) = S_K(\theta) + O(\epsilon nK).$$
We choose $\epsilon = \frac{\delta}{C nK}$, so that 
$$\sum_{a \in A} S_K(a(\theta + t) ) = \sum_{a \in A} S_K(a\theta) + O(\delta C^{-1}).$$
Choosing $C = C(\delta)$ large enough, we find that 
$$\mu(B) \geq \exp(- K n  \log (CnK)),$$
and $\theta + B \subset E_{\alpha }$, as desired. 
\end{proof}

\begin{lemma}\label{intSK} Suppose $\delta > 0$ and
$$\sum_{a \in A} S_K(ax) \geq M,$$
for some $M = O(1)$. Then for $C = C(\delta)$ large enough and some $k \leq  C Kn^2\log(CnK)$, we have $$ \left( \int \left( \sum_{a\in A} (S_K(ax) + m_K)\right)^{2k} \right)^{1/(2k)}  >  (M+nm_K-\delta).$$
\end{lemma}

\begin{proof}
Let $k \geq C Kn^2\log(CnK)$. By Lemma~\ref{meas} and \eqref{mx}, we have
\begin{align*} \label{int} \left( \int \left( \sum_{a\in A} (S_K(ax) + m_K)\right)^{2k} \right)^{1/(2k)} & \geq \exp(-1/(Cn)) (M + nm_K- \delta/2 ) \\
&\geq M+ nm_K-  \delta /2+ O(C^{-1}).\end{align*} 
\end{proof}

Thus we have shown that if 
$$\sum_{a \in A} S_K(ax),$$ is large for some $x \in \mathbb{T}$, then the integral in Lemma~\ref{intSK} must also be large. We will now proceed to show that, using \eqref{FSK}, that, up to some technical issues, this integral
$$\int (S_K +m_K)^{2k},$$ is invariant under Freiman isomporphism of sufficiently large order. We recall a Freiman isomorphism \cite[Definition 5.2.1]{TV}.
\begin{defn}[Frieman isomorphism] Let $M \geq 1$ and $A ,B$ be subsets of some abelian groups. A Freiman isomorphism of order $M$ is a bijective map $\phi : A \to B$ with the property that for all $a_1,a'_1 , \ldots ,a_m ,a'_m\in A$,
$$a_1 + \cdots a_M = a'_1 + \cdots + a'_M,$$
if and only if 
$$\phi(a_1) + \cdots + \phi(a_M) = \phi(a'_1) + \cdots + \phi(a'_M).$$
We say $\phi$ fixes zero if $0 \in A$ and $\phi(0) = 0$. 

\end{defn}

\begin{thm}[\cite{KL}]\label{KL} Let $A \subset \mathbb{Z}_{\neq 0}$ of size $n$. Then there is a Freiman ismorphism, $\phi$ of order $M$ of  $A \cup \{0\}$ that fixes 0 such that $\phi(A) \subset \mathbb{Z}_{\neq 0}$ and $$\max_{a \in A} |a| \leq (8M)^n.$$

\end{thm}
We include a proof for completeness (as we need to fix 0), which we learned of from Lev and differs from the approach in \cite{KL}.
\begin{proof} We let $A= \{a_1 , \ldots a_n\}$, $\ell := \max_{a \in A} |a|$ and $4M\ell < p \leq 8M \ell$ be a prime, which exists by Bertrand's postulate. We may suppose $p > (4M)^n$, since otherwise we are done. Let $A_p$ be the canonical image of $A$ in $\mathbb{Z}/p\mathbb{Z}$.  Consider the $n$-dimensional lattice
$$\Lambda := (a_1 , \ldots , a_n) \mathbb{Z} + p \mathbb{Z}^n.$$
This has determinant $p^{n-1}$ and so by Minkowski's first theorem there is a non-zero point
$$(y_1 , \ldots , y_n) \in [-p^{1-1/n} , p^{1-1/n}] \cap \Lambda.$$
Thus $\{y_1 , \ldots , y_n\} \subset \mathbb{Z} / p \mathbb{Z}$ is $M$ Freiman isomorphic to $A$, fixing 0. We now map $y_i$ to the integer $-p/2 < z_i < p/2$. As $p^{1-1/n} 2M < p/2$, we have this is a Freiman isomporphsim of order $M$. Thus we have decreased the $\max_{a \in A} |a|$, unless $p^{1-1/n}  \geq \ell$. Continuing in this way, we may assume $$(8M\ell)^{1-1/n} \geq p^{1-1/n}  \geq  \ell,$$
and so $$\ell \leq (8M)^{n}.$$
\end{proof}

We are now ready to prove Theorem~\ref{main2}.

\begin{proof}[Proof of Theorem~\ref{main2}]

We apply Lemma~\ref{KL} with $M=2kK$ and let $B$ be the $M$-Freiman isomporphic set. As $A \cap - A = \emptyset$, we have $B \cap -B = \emptyset$. We may replace the negative elements of $B$ with the positive elements without altering 
$$\sum_{b \in B} S_K(bx),$$
so we assume all the elements of $B$ are positive. By the assumption of Theorem~\ref{main2}, we have 
$$\sum_{b \in B} S_K(bx) \geq S/3 + \delta,$$
By Lemma~\ref{intSK}, this implies
 \begin{align*} m_A + nm_K& \geq \left( \int \left( \sum_{a\in A} (S_K(ax) + m_K)\right)^{2k} \right)^{1/(2k)} \\  & = \left( \int \left( \sum_{b\in B} (S_K(bx) + m_K)\right)^{2k} \right)^{1/(2k)} 
  \geq (S/3+nm_K ),\end{align*}
 as desired. 
\end{proof}

\end{document}